\documentclass[a4paper]{article}
\usepackage{amsmath,amssymb,graphicx}

\newtheorem{prethm}{{\bf Theorem}}

\newenvironment{thm}{\begin{prethm}{\hspace{-0.5
               em}{\bf.}}}{\end{prethm}}

\newtheorem{prepro}[prethm]{Proposition}
\newenvironment{pro}{\begin{prepro}{\hspace{-0.5
               em}{\bf.}}}{\end{prepro}}

\newtheorem{preconj}[prethm]{Conjecture}
\newenvironment{conj}{\begin{preconj}{\hspace{-0.5
               em}{\bf.}}}{\end{preconj}}

\newtheorem{prelem}[prethm]{Lemma}
\newenvironment{lem}{\begin{prelem}{\hspace{-0.5
               em}{\bf.}}}{\end{prelem}}

\newtheorem{precor}[prethm]{Corollary}

\newtheorem{prerem}[prethm]{{\bf Remark}}

\newtheorem{preexample}{{\bf Example}}

\newtheorem{preproof}{{\bf Proof.}}

\newenvironment{proof}[1]{\begin{preproof}{\rm
               #1}\hfill{$\Box$}}{\end{preproof}}

\newcommand{\e}{\epsilon}

\newcommand{\al}{\alpha}

\renewcommand{\thefootnote}

\title{On two conjectures on sum of the powers of signless Laplacian eigenvalues of a graph}

\author{ F. Ashraf \vspace{.4cm}\\
{\sl Department of Mathematical Sciences, Isfahan University of Technology,}\\
{\sl Isfahan, 84156-83111, Iran}\\
{\sl School of Mathematics, Institute
for Research in Fundamental Sciences (IPM),} \\
 {\sl  P.O. Box 19395-5746, Tehran, Iran}\\
}
\date{}
\begin{document}
\maketitle
\footnotetext{ E-mail Address: {\tt firouzeh\_ashraf@yahoo.com}}

\begin{abstract}
Let $G$ be a simple graph and $Q(G)$ be
  the signless Laplacian matrix of $G$. Let $S_\al(G)$ be the sum of the $\al$-th powers of the nonzero eigenvalues of $Q(G)$.
We disprove two conjectures by You and Yang on the extremal values of $S_\al(G)$ among bipartite graphs and among
graphs with bounded connectivity.

\vspace{3mm}
\end{abstract}

\section{Introduction}

Let $G$ be a simple graph with vertex set $V(G) = \{v_1,\ldots, v_n\}$. The degree of a vertex $v\in V(G)$,
denoted by $d(v)$, is the number of neighbors of $v$.
The {\em adjacency matrix} of $G$ is an $n\times n$ matrix $A(G)$ whose $(i,j)$ entry is 1 if $v_i$ and $v_j$ are adjacent and zero otherwise.
The {\em signless Laplacian matrix} of $G$ is the matrx
$Q(G) =A(G)+D(G)$,
where $D(G)$ is the diagonal matrix with $d(v_1),\ldots, d(v_n)$ on its main diagonal.
It is well-known that $Q(G)$ are positive semidefinite
and so its eigenvalues are nonnegative real numbers. The  multiplicity of zero eigenvalue for $Q(S)$ is equal to the number of bipartite connected components of $G$.
The eigenvalues of  $Q(G)$ are called the
 {\em signless Laplacian eigenvalues} of $G$ and are denoted by %$\mu_1(G)\ge\cdots\ge\mu_n(G)$ and
$q_1(G),\ldots,q_n(G)$. We drop $G$ from the notation when there is no danger of confusion.
%Note that each row sum of $L(G)$ is 0 and therefore, $\mu_n(G) = 0$.
We denote the complete graph on $n$ vertices by $K_n$ and the complete bipartite graph with parts with $r$ and $s$ vertices by $K_{r,s}$.
The (vertex) connectivity $\kappa(G)$  of a connected graph $G$ is the minimum number of vertices of $G$ whose deletion disconnects
$G$. It is conventional to define $\kappa(K_n) = n-1$. For two graphs $G$ and $H$, the {\em join} of them denoted by $G\vee H$ is the graph obtained from  disjoint union of $G$ and $H$ by adding edges joining every vertex
of $G$ to every vertex of $H$. We also denote the number of edges of $G$ by $e(G)$.

For a graph $G$, let $q_1(G),\ldots,q_r(G)$ be all the nonzero signless Laplacian eigenvalues of $G$.
  You and Yang \cite{yy} studied the parameter
$$S_\al(G):=q_1(G)^\al+\cdots+q_r(G)^\al.$$
Among other things, they obtained the following two results.
\begin{thm}\label{1} {\rm(\cite{yy})}
Let $G$ be a connected bipartite graph with $n$ vertices and $\alpha\leq1$.
\begin{itemize}
  \item[\rm(i)]If $\alpha<0$, then $S_{\alpha}(G)\geq n^\alpha+\left(\lfloor n/2\rfloor-1\right)\lceil n/2\rceil^\alpha+(\lceil n/2\rceil-1)\lfloor n/2\rfloor^\alpha,$
with equality if and only if $G=K_{\lfloor n/2\rfloor,\lceil n/2\rceil}$.
\item[\rm(ii)] If  $0<\alpha\leq1$, then $S_{\alpha}(G)\leq n^\alpha+\left(\lfloor n/2\rfloor-1\right)\lceil n/2\rceil^\alpha+(\lceil n/2\rceil-1)\lfloor n/2\rfloor^\alpha,$
with equality if and only if $G=K_{\lfloor n/2\rfloor,\lceil n/2\rceil}$.
\end{itemize}
\end{thm}
%\begin{rem} The connectedness condition in part (i) of Theorem~\ref{1} was missed in \cite{yy}. We notice that without this condition the statement is not correct. For example, consider $K_2\cup2K_1$ and $K_{2,2}$ whose signless Laplacian eigenvalues are
%$2,0,0,0$ and $4,2,2,0$, respectively (see Lemma~\ref{QK_n} below). Hence $S_{-1}(K_2\cup2K_1)<S_{-1}(K_{2,2})$.
%\end{rem}
\begin{thm}\label{2} {\rm(\cite{yy})}
Let $G$ be a connected graph with $n$ vertices and $\kappa(G)\le k$ and $\alpha\geq1.$ Then
$S_\alpha(G)\leq b_\alpha(n,k)$
where
\begin{align*}
  b_\al(n,k)&=k(n-2)^\alpha+(n-k-2)(n-3)^\alpha+\left(n-2+\frac{k}{2}+\frac{1}{2}\sqrt{(k-2n)^2+16(k-n+1)}\right)^\alpha\\
  &~~~~~+\left(n-2+\frac{k}{2}-\frac{1}{2}\sqrt{(k-2n)^2+16(k-n+1)}\right)^\alpha.
\end{align*}
 The equality  holds if and only if $G=K_k\vee(K_1\cup K_{n-k-1}).$
\end{thm}

For the unsettled values of $\al$ in Theorems~\ref{1} and \ref{2}, they  made the following two conjectures.
\begin{conj}\label{BipPow}  {\rm(\cite{yy})}
Let $G$ be a bipartite graph with $n$ vertices. If $\alpha>1,$ then  $$S_{\alpha}(G)\leq n^\alpha+\left(\lfloor n/2\rfloor-1\right)\lceil n/2\rceil^\alpha+(\lceil n/2\rceil-1)\lfloor n/2\rfloor^\alpha,$$
with equality if and only if $G=K_{\lfloor n/2\rfloor,\lceil n/2\rceil}$.
\end{conj}

\begin{conj}\label{k-con} {\rm(\cite{yy})}
 %$n,$ $k$ be positive integers with $1\leq k\leq n-1,$
Let $G$ be a graph with $n$ vertices and $\kappa(G)\le k$.
\begin{itemize}
  \item[\rm(i)] If $0<\alpha<1,$ then
$S_\alpha(G)\leq b_\alpha(n,k)$ with equality if and only if $G=K_k\vee(K_1\cup K_{n-k-1}).$
  \item[\rm(ii)] If $G$ is connected and $\alpha<0,$ then $S_\alpha(G)\geq b_\alpha(n,k)$ with equality if and only if $G=K_k\vee(K_1\cup K_{n-k-1}).$
\end{itemize}
\end{conj}

The purpose of this paper is to study these two conjectures. We prove the following results in this regard:
\begin{itemize}
\item For $\al>0$, we determine
$$\lim_{n\to\infty}\frac{\max\left\{S_{\alpha}(G)\mid \hbox{$G$ is a bipartite graph with $n$ vertices}\right\}}{n^{\al+1}}$$
from which it follows that Conjecture~\ref{BipPow} is not true for $\al>3$;
\item Conjecture~\ref{BipPow} is true for  $1\le\al\le3$;
\item Conjecture~\ref{k-con} is not true for $\al<-1$.
\end{itemize}
The validity of Conjecture~\ref{k-con} for $-1\le\al\le1$ remains open.

\section{Bipartite graphs}

In this section we study the asymptotic behavior of  the function
$$\zeta(n,\al):=\max\left\{S_{\alpha}(G)\mid \hbox{$G$ is a bipartite graph with $n$ vertices}\right\},$$
for $\al>0$. We start with the following well-known fact.
\begin{lem}\label{inter} {\rm(\cite[p. 222]{crsB})} Let $G$ be a graph and $e$ be an edge of that. Then the signless Laplacian eigenvalues of
 $G$ and $G'=G-e$ interlace:
$$q_1(G)\ge q_1(G')\ge q_2(G)\ge q_2(G')\ge\cdots\ge q_n(G)\ge q_n(G').$$
\end{lem}

The following lemma is easy to prove.
\begin{lem}\label{QK_n}
\begin{itemize}
  \item[\rm(i)] The signless Laplacian eigenvalues of $K_n$ are $2n-2$ with multiplicity $1$ and $n-2$ with multiplicity $n-1$.
  \item[\rm(ii)] The signless Laplacian eigenvalues of $K_{r,s}$ are $r+s$ with multiplicity $1$, $r$ with multiplicity $s-1$, $s$ with multiplicity $r-1$, and $0$ with multiplicity $1$.
\end{itemize}
\end{lem}

For the next theorem, we need Taylor Theorem which we recall here.
If the $k$-th derivative of a real function $f$ exists on an interval containing $a$ and $a+\e$, then there exists some $\eta$ between
$a$ and $a+\e$ such that
$$f(a+\e)=f(a)+f'(a)\e+\frac{f''(a)}{2!}\e^2+\cdots+\frac{f^{(k-1)}(a)}{(k-1)!}\e^{k-1}+\frac{f^{(k)}(\eta)}{k!}\e^k.$$

In the next theorem, we determine the asymptotic behavior of $\zeta(n,\al)$.  Noting that the upper bound given in Conjecture~\ref{BipPow} is $2^{-\al}n^{\al+1}+O(n^\al)$, the next theorem disproves Conjecture~\ref{BipPow} for $\al>3$.

\begin{thm} For any $\al>0$,
$$\lim_{n\to\infty}\frac{\zeta(n,\al)}{n^{\al+1}}=p(\al)$$
where
$$p(\al)=\max\{x(1-x)^\al+(1-x)x^\al\mid 0\le x\le1\}.$$
Furthermore, for any $\al>3$, we have $p(\al)>2^{-\al}$.
\end{thm}
\begin{proof}{For a bipartite graph $G$ with parts of sizes $r$ and $n-r$, by Lemma~\ref{inter}, we have $S_\al(G)\le S_\al(K_{r,n-r})$. Therefore the maximum occurs for some $K_{r,n-r}$, i.e. for any $n$ there exists some $r$ for which
$\zeta(n,\al)=S_\al(K_{r,n-r})$.
We now fix $\al$ and let $$f(x):=x(1-x)^\al+(1-x)x^\al.$$
By Lemma~\ref{QK_n},
\begin{align}
   S_\al(K_{r,n-r})&=n^\al+(r-1)(n-r)^\al+(n-r-1)r^\al\nonumber\\
   &=\left[\frac{r}{n}\left(1-\frac{r}{n}\right)^\al+\left(1-\frac{r}{n}\right)
   \left(\frac{r}{n}\right)^\al\right]n^{\al+1}+O(n^\al)\nonumber\\
   &=f\left(\frac{r}{n}\right)n^{\al+1}+O(n^\al)\label{SK}.
\end{align}
It follows that for large enough $n$,
\begin{equation}\label{<}
\frac{\zeta(n,\al)}{n^{\al+1}}\le p(\al)+o(1).
\end{equation}
Now we choose $0<b<1$ so that $f(b)=p(\al)$. Let $r_n=\lfloor bn\rfloor$. From (\ref{SK}), for large enough $n$ we have
\begin{align}\label{>}
\frac{\zeta(n,\al)}{n^{\al+1}}&\ge\frac{S_\al(K_{r_n,n-r_n})}{n^{\al+1}}\nonumber\\
&\ge f\left(\frac{\lfloor bn\rfloor}{n}\right)+o(1).
\end{align}
%\frac{\lfloor bn\rfloor}{n}\left(1-\frac{\lfloor bn\rfloor}{n}\right)^\al+\left(1-\frac{\lfloor %bn\rfloor}{n}\right)\left(\frac{\lfloor bn\rfloor}{n}\right)^\al
Combining (\ref{<}) and (\ref{>}), and then taking the limit, shows that $\lim_{n\to\infty}\zeta(n,\al)/n^{\al+1}$ exists and equals to $p(\al)$.

For the second part of the theorem, we fix $\al>3$.
Not that since $\al>3$, we have ${\al\choose2}-\al>0$. So we may choose $0<\e<1/2$ small enough so that
\begin{equation}\label{e2-e4}
\left[{\al\choose2}-\al\right]\e^2(1/2)^{\al-2} -2{\al\choose3}\e^4>0.
\end{equation}
We will show that  by this choice of $\e$, one has $f(1/2+\e)>f(1/2)=2^{-\al}$, and consequently $p(\al)>2^{-\al}$.

By applying Taylor Theorem for $f(x)$ with $k=3$ and $a=1/2$, there exit $\eta_1,\eta_2$ with
$\frac{1}{2}-\e<\eta_1<\frac{1}{2}<\eta_2<\frac{1}{2}+\e$ such that
\begin{align*}
   (1/2-\e)^\al&=(1/2)^\al-\al\e(1/2)^{\al-1}+{\al\choose2}\e^2(1/2)^{\al-2}-{\al\choose3}\e^3\eta_1^{\al-3},\\
   (1/2+\e)^\al&=(1/2)^\al+\al\e(1/2)^{\al-1}+{\al\choose2}\e^2(1/2)^{\al-2}+{\al\choose3}\e^3\eta_2^{\al-3}.
\end{align*}
It follows that
\begin{align*}
   f(1/2+\e)&=(1/2+\e)(1/2-\e)^\al+(1/2-\e)(1/2+\e)^\al\\
    &=(1/2)^\al+{\al\choose2}\e^2(1/2)^{\al-2}+{\al\choose3}\frac{\e^3}{2}(\eta_2^{\al-3}-\eta_1^{\al-3})
     -2\al\e^2(1/2)^{\al-1}-{\al\choose3}\e^4(\eta_1^{\al-3}+\eta_2^{\al-3}).
\end{align*}
Note that
\begin{align}
    {\al\choose2}\e^2(1/2)^{\al-2}-2\al\e^2(1/2)^{\al-1}&-{\al\choose3}\e^4(\eta_1^{\al-3}+\eta_2^{\al-3})\nonumber\\
    &=\left[{\al\choose2}-\al\right]\e^2(1/2)^{\al-2}-{\al\choose3}\e^4(\eta_1^{\al-3}+\eta_2^{\al-3}).\label{e2-e4eta}
\end{align}
As $\eta_1^{\al-3}+\eta_2^{\al-3}<2$, from (\ref{e2-e4}) it follows that  the right side of (\ref{e2-e4eta}) is positive.
This implies that $f(1/2+\e)>(1/2)^\al$, as desired.
}\end{proof}

\begin{thm} Conjecture~{\rm\ref{BipPow}} is true for $1\le\al\le3$.
\end{thm}
\begin{proof}{Let $$g(x):=(x-1)(n-x)^\al+(n-x-1)x^\al.$$
 Then $S_\al(K_{r,n-r})=n^\al+g(r)$.
We prove the theorem by showing that for $1\le\al\le3$ and for any $1\le r\le n-1$, $g(r)\le g(\lfloor n/2\rfloor)$.
Since $g(x)=g(n-x)$, we may assume that $1\le x\le n/2$. So it suffices to show that $g$ is increasing on the interval
$0< x\le n/2$.

We have
\begin{align*}
g'(x)&=(n-x)^\al-\al(x-1)(n-x)^{\al-1}-x^\al+\al(n-x-1)x^{\al-1}\\
  &=x^\al\left[\left(\frac{n}{x}-1\right)^\al-\al\left(1-\frac{1}{x}\right)\left(\frac{n}{x}-1\right)^{\al-1}-1+\al
  \left(\frac{n}{x}-1-\frac{1}{x}\right)\right].
\end{align*}
Since $n/x\ge2$, we see  $\frac{n}{x}-1-\frac{1}{x}\ge\left(1-\frac{1}{x}\right)\left(\frac{n}{x}-1\right)$.

First assume that $1<\al\le2$. So $\left(\frac{n}{x}-1\right)\ge\left(\frac{n}{x}-1\right)^{\al-1}$.
Therefore,
$$\frac{n}{x}-1-\frac{1}{x}\ge\left(1-\frac{1}{x}\right)\left(\frac{n}{x}-1\right)^{\al-1}.$$
This together with $(n/x-1)^\al\ge1$ imply that $g'(x)\ge0$ for $0<x\le n/2$ and so $g$ is increasing.

Next, assume that $2<\al\le3$. We have
$$g''(x)=-2\al\left[(n-x)^{\al-1}+x^{\al-1}\right]+\al(\al-1)\left[(x-1)(n-x)^{\al-2}+(n-x-1)x^{\al-2}\right].$$
Note that since $0<x\le n/2$,
$(n-x)^{\al-2}(n-2x+1)> x^{\al-2}(n-2x-1)$ which implies that $$(n-x)^{\al-1}-(x-1)(n-x)^{\al-2}>(n-x-1)x^{\al-2}-x^{\al-1}.$$
So we have  $$(n-x)^{\al-1}+x^{\al-1}>(x-1)(n-x)^{\al-2}+(n-x-1)x^{\al-2}.$$
Since $1<\al\le3$, $2\al>\al(\al-1)$ and so it follows that $g''(x)<0$ for $0<x\le n/2$.
Hence $g'$ is decreasing, and so $g'(x)\ge g'(n/2)=0$, and again we are done.
}\end{proof}

\section{Graphs with bounded connectivity}

In this section we consider $S_\al(G)$ for graphs $G$ with bounded connectivity and disprove Conjecture~\ref{k-con} for $\al<-1$.
Let $G$ be an $n$-vertex graph with $\kappa(G)\le k$.
Then $G$ must be a subgraph of one of  the graphs
$K_k\vee(K_r\cup K_{n-k-r})$ for some $r=1,\ldots,\lfloor(n-k)/2\rfloor$.
In view of Lemma~\ref{inter}, it follows that (as observed in \cite{yy}) the extremal values of $S_\al(G)$ correspond to one of the graphs
$K_k\vee(K_r\cup K_{n-k-r})$ for some $r\in\{1,\ldots,\lfloor(n-k)/2\rfloor\}$.
We first compute the signless Laplacian eigenvalues of these graphs.

For a graph $G$, consider a partition $P=\{V_1,\ldots,V_m\}$ of $V(G)$.
The partition of $P$ is {\em equitable} if each submatrix $Q_{ij}$ of $Q(G)$ formed by the
rows of $V_i$ and the columns of $V_j$ has constant row sums $r_{ij}$.
The $m\times m$ matrix $R=(r_{ij})$ is called the {\em quotient matrix}
of $Q(G)$ with respect to $P$. The proof of the following theorem is similar to the one given in \cite[p. 187]{crsB} where a similar result is presented for Laplacian matrix.

\begin{lem}\label{equi}
Any eigenvalue of the quotient matrix $R$ is an eigenvalue of $Q(G)$.
\end{lem}

\begin{lem}\label{Kkr}
The signless Laplacian eigenvalues of  $K_k\vee(K_r\cup K_{n-k-r})$ for $1\le k\le n-2$ and $1\le r\le (n-k)/2$ are
$$(n-2)^{[k]},\, (k+r-2)^{[r-1]}, \,(n-r-2)^{[n-k-r-1]},\,
n-2+\frac{k}{2}\pm\frac{1}{2}\sqrt{(k-2n)^2+16r(k-n+r)},$$
where the exponents indicate multiplicities.
\end{lem}
\begin{proof}{Let $G=K_k\vee(K_r\cup K_{n-k-r})$. The partition of $V(G)$ into the vertex sets of the subgraphs
$K_k,K_r,K_{n-k-r}$ forms an equitable partition of $Q(G)$. The corresponding quotient matrix is
$$\left(\begin{array}{ccc}n+k-2&r&n-k-r\\k&2r+k-2&0\\k& 0&2(n-r-1)-k\end{array}\right),$$
with eigenvalues $n-2,\, n-2+\frac{k}{2}\pm\frac{1}{2}\sqrt{(k-2n)^2+16r(k-n+r)}$.

To determine the rest of the eigenvalues, note that in the matrices $Q(G)-(n-2)I$, $Q(G)-(k+r-2)I$ and $Q(G)-(n-r-2)I$, the rows corresponding to the vertices of $K_k$, $K_r$  and $K_{n-k-r}$, respectively, are identical. It follows that the nullities of the matrices
$Q(G)-(n-2)I$, $Q(G)-(k+r-2)I$ and $Q(G)-(n-r-2)I$, are at least $k-1$, $r-1$  and $n-k-r-1$, respectively.
Therefore $n-2$, $k+r-2$ and $n-r-2$ are eigenvalues of $Q(G)$ with multiplicities at least $k-1$, $r-1$  and $n-k-r-1$, respectively.
So far we have obtained $n-1$ eigenvalues of $Q(G)$.
To determine the remaining eigenvalue we use the fact that
the sum of all eigenvalues of $Q(G)$ equals $2e(G)$; it turns out that the remaining eigenvalue is also $n-2$. The proof is now complete.
}\end{proof}

The next proposition disproves  Conjecture~\ref{k-con} for  $\al<-1$.

\begin{pro} For any $\al<-1$, any positive integer $k$ and for large enough $n$, there exist $k$-connected graphs $G$ with $n$ vertices such that
 $S_\al(G)<b_\al(n,k)$.
\end{pro}
\begin{proof}{
Note that
$$\lim_{n\to\infty}\left(n-2+\frac{k}{2}-\frac{1}{2}\sqrt{(k-2n)^2+16(k-n+1)}\right)=k.$$
For $\al<-1$, the other terms of  $b_\al(n,k)$ tends to zero as $n\to\infty$. Hence
$$\lim_{n\to\infty}b_\al(n,k)=k^\al.$$
On the other hand, by Lemma~\ref{Kkr}, $S_\al\left(K_k\vee(K_{(n-k)/2}\cup K_{(n-k)/2})\right)$ equals to
$$k(n-2)^\al+\frac{1}{2}(n-k-2)(n+k-4)^\al+\left(n-2+\frac{k}{2}+\frac{1}{2}\sqrt{4kn-3k^2}\right)^\al+
\left(n-2+\frac{k}{2}-\frac{1}{2}\sqrt{4kn-3k^2}\right)^\al.$$
It is seen that for $\al<-1$,
$$\lim_{n\to\infty}S_\al\left(K_k\vee(K_{(n-k)/2}\cup K_{(n-k)/2})\right)=0.$$
This means that for any positive integer $k$ and for large enough $n$,
$$S_\al\left(K_k\vee(K_{(n-k)/2}\cup K_{(n-k)/2})\right)<b_\al(n,k).$$
}\end{proof}

\section*{Acknowledgments}
The author thanks Dr. B. Tayfeh-Rezaie for his comments on the manuscript.

\end{document}